\documentclass{amsart}

\usepackage{latexsym}
\usepackage{amsmath,amssymb,stmaryrd}
\usepackage{amscd,amsxtra}
\usepackage{amsthm} 
\usepackage[all]{xy}

\theoremstyle{plain} 
\newtheorem{thm}{\bf Theorem}[section]
\newtheorem{pro}[thm]{\bf Proposition}
\newtheorem{lem}[thm]{\bf Lemma}

\newtheorem{rem}[thm]{\bf Remark}
\newtheorem{open}[thm]{\bf Open Problem}
\newtheorem{ex}[thm]{\bf Example}

\DeclareMathOperator{\Hom}{Hom}
\DeclareMathOperator{\Ext}{Ext}
\DeclareMathOperator{\Ann}{Ann}
\DeclareMathOperator{\Proj}{Proj}

\DeclareMathOperator{\im}{Im}
\DeclareMathOperator{\Spec}{Spec}
\DeclareMathOperator{\cM}{\mathcal{M}}
\DeclareMathOperator{\Char}{char}
\DeclareMathOperator{\E}{E}

\title{On Lyubeznik numbers of projective schemes}
\author{Wenliang Zhang}
\begin{document}

\begin{abstract}
Let $X$ be an arbitrary projective scheme over a field $k$. Let $A$ be the local ring at the vertex of the affine cone for some embedding $\iota: X\hookrightarrow \mathbb{P}^n_k$. G. Lyubeznik asked (in \cite{l2}) whether the integers $\lambda_{i,j}(A)$ (defined in \cite{l1}), called the Lyubeznik numbers of $A$, depend only on $X$, but not on the embedding. In this paper, we  make a big step toward a positive answer to this question by proving that in positive characteristic, for a fixed $X$, the Lyubezink numbers 
$\lambda_{i,j}(A)$ of the local ring $A$, can only
achieve finitely many possible values under all choices of embeddings. 
\end{abstract}

\maketitle

\section{Introduction}
Let $A$ be a local ring that contains a field and admits a surjection from a $n$-dimensional regular 
local ring $(R,\mathfrak{m})$ containing a field. Let $I\subset R$ be the kernel of the surjection, and $k=R/{\mathfrak{m}}$ the residue field of $R$. 
The Lyubeznik numbers $\lambda_{i,j}(A)$ (Definition 4.1 in \cite{l1}) 
are defined to be $\dim_k(\Ext_R^i(k,H^{n-j}_I(R)))$. And it was proven in \cite{l1} that they are all finite and depend 
only on $A,i$ and $j$, but neither on $R$, nor on the surjection $R\rightarrow A$. \par

The Lyubeznik numbers have been studied by many authors; see, for example, \cite{bb}, \cite{ls}, \cite{k1}, \cite{k2},  \cite{l3}, \cite{w}, and \cite{zh}. In particular, in \cite{zh}, it is proven that, for a $d$-dimensional projective
 scheme $X$, where $A$ is the local ring at the vertex of the affine cone for some embedding $\iota:X\hookrightarrow \mathbb{P}^n_k$, the highest Lyubeznik number $\lambda_{d+1,d+1}(A)$ depends only on $X$ but not on the embedding. This gives some supporting evidence for a positive answer to the following open problem posed in \cite{l2}

\begin{open}
\label{open1}
Let $X$ be a projective scheme over a field $k$, and let $A$ be the local ring at the vertex of the affine cone over $X$ for some 
embedding of $X$ into a projective space. Is it true that $\lambda_{i,j}(A)$ depend only on $X$ but not on the embedding?
\end{open}

In this paper, we provide some strong supporting evidence to this open problem. In fact we come very close to a proof that in positive characteristic the answer to this open problem is indeed positive.\par
Let $R$ be a Noetherian ring of positive characteristic $p$ and let $k$ be a coefficient field of $R$ which for the purpose of this introduction we assume to be perfect. Let $M$ be an $R$-module equipped with an action of Frobenius $f$, i.e., a map of abelian groups $f:M\to M$ such that $f(rm)=r^pf(m)$ for all $r\in R,m\in M$. The stable part of $M$, denoted by $M_s$, is defined to be 
$$\bigcap ^{\infty}_{i=1}f^i(M)$$
(in general, this is not an $R$-module, but a $k$-vector-space). The operation of `taking the stable part' has played an important role in the study of local cohomology in positive characteristic, for example, it has been used to estimate local cohomological dimensions (cf. \cite{lcd}) and to prove vanishing results of local cohomology (cf. \cite{l5}), etc. In this paper, we will give a description of $\lambda_{i,j}(A)$ ($A$ as in Open Problem \ref{open1}) in terms of the stable part of a certain module $\cM$ (defined in the next paragraph).\par

Let $\iota:X\hookrightarrow \mathbb{P}^n_k$ be an embedding of $X$. Let $d$ denote the dimension of $X$. Let $R=k[x_0,\cdots,x_n]$ and $I$ be the defining ideal of $X$ ($I$ is a homogeneous ideal of $R$). Let $$\cM=\Ext^{n+1-i}_R(\Ext^{n+1-j}_R(R/I,R),R).$$ There is a natural action of Frobenius $f$ on $\cM$ (this is explained in detail in Proposition \ref{f-on-ext}). The following theorem is the main theorem of this paper.
\newtheorem*{mthm}{\bf Main Theorem}
\begin{mthm}
Let $R,I,\cM,d$ be as above and $A=(R/I)_{(x_0,\cdots,x_n)}$ be the local ring at the vertex of the affine cone over $X$ under the embedding $\iota:X\hookrightarrow \mathbb{P}^n_k$. Then
$$\lambda_{i,j}(A)=\dim_k(\cM_s),$$
for $0\leq i,j\leq d+1$.
\end{mthm}

Since $\cM$ is naturally graded and the action of Frobenius satisfies $$\deg(f(m))=p\deg(m)$$ (this is explained in detail in Section 3), $\cM_s$ is contained in $\cM_0$, the degree-0 piece of $\cM$, which is a finite dimensional $k$-vector-space, and hence $$\dim_k(\cM_s)\leq \dim_k(\cM_0),$$ i.e., 
$$\lambda_{i,j}(A)\leq\dim_k(\cM_0).$$ But $\cM_0$ does not depend on the embedding (this key fact is proved in Theorem \ref{inv}), hence we have 
\newtheorem*{mcor}{\bf Corollary}
\begin{mcor} 
Let $X$ be a $d$-dimensional projective scheme over a field $k$ of characteristic $p>0$. Let $A$ be as in Open Problem \ref{open1}. Then, $\lambda_{i,j}(A)$ can only achieve finitely many possible values for all choices of
embeddings, for all $0\leq i,j\leq d+1$.
\end{mcor}
This corollary provides some strong supporting evidence for a positive answer to Open Problem \ref{open1}. In fact, our Main Theorem reduces Open Prolem \ref{open1} to the following.

\begin{open}
\label{open2}
Let $\cM$ be as above. Is the restriction of the natural action of Frobenius on $\cM$ to $\cM_0$ independent of the embedding?
\end{open}

We believe the answer to Open Problem \ref{open2} is indeed positive, but this remains to be proven. A positive answer to Open Problem \ref{open2} would immediately imply a positive answer to Open Problem \ref{open1} via our Main Theorem. 

This paper is organized as follows: in Section 2, we prove some characteristic-free results, especially, we prove that $\cM_0$ ($\cM$ is as above) does not depend on the embedding; in Section 3, we study modules with actions of Frobenius and draw connections between $\lambda_{i,j}(A)$ and the module $\cM$, especially, we prove the Main Theorem.   
\section{Some characteristic-free results}
In this section we will prove some characteristic-free results which are crucial to the proof of our Main Theorem. The key result is Theorem \ref{inv}, which says, if $X=\Proj(k[x_0,\cdots,x_n]/I)$ under the embedding $\iota:X\hookrightarrow \mathbb{P}^n_k$, where $k$ is a field of any characteristic, then the degree-0 piece of the module $$\cM=\Ext^{n+1-i}_R(\Ext^{n+1-j}_R(R/I,R),R),$$
where $R=k[x_0,\dots,x_n]$, depends only on $X$, but not on the embedding $\iota$.\par
To this end, we need to recall some notations and results from \cite{rd}. A definition of an embeddable morphism is stated on page 189 in \cite{rd}. All we need to know for our purposes is that a smooth morphism is embeddable, a finite morphism is embeddable and a composition of embeddable morphisms is embeddable. 
For an embeddable morphism of locally Noetherian schemes $$f: Y_1\rightarrow Y_2,$$ 
there exists a functor (Theorem 8.7 in Chapter 3 in \cite{rd}) 
$$f^!:\mathcal{D}^+_{qc}(Y_2)\rightarrow \mathcal{D}^+_{qc}(Y_1),$$
(where $\mathcal{D}^+_{qc}(Y)$ denotes the derived category of bounded below complexes of quasi-coherent sheaves of $\mathcal{O}_Y$-modules for any scheme $Y$), satisfying 
\begin{enumerate}
\item if there are two consecutive embeddable morphisms, $Y_1\xrightarrow{f}Y_2\xrightarrow{g}Y_3$, then, by Theorem 8.7(2) in Chapter 3 in \cite{rd}, one has
$$(gf)^!=f^!g^!.$$ 
\item if $f:Y_1\to Y_2$ is a smooth morphism, then 
$$f^!(\mathcal{O}_{Y_2})=\omega_{Y_1/Y_2}$$
by the Remark on page 143 in \cite{rd}.
\item if $f:Y_1\to Y_2$ is a finite morphism, then
$$f^!(\cdot)=\bar{f}^*\underline{\underline{R}}\ \underline{\Hom}_{\mathcal{O}_{Y_2}}(f_*\mathcal{O}_{Y_1}, \cdot),$$
where $\bar{f}$ denotes the induced morphism 
$$(Y_1,\mathcal{O}_{Y_1})\rightarrow (Y_2, f_*\mathcal{O}_{Y_1}),$$
(see Definition on page 165 in \cite{rd} and Theorem 8.7(3) in Chapter 3 in \cite{rd}).
\end{enumerate}

The following theorem should be well-known to experts. It was stated without proof as the last assertion of Proposition 5 in \cite{gr} (the proof of Proposition 5 in \cite{gr} says ``la derni\`ere assertion de la proposition 5 est plus subtile, et ... ne peut \^{e}tre donn\'e ici''). Since we could not find a proof in the literature, we give a proof here.

\begin{thm}
Let $X$ be an arbitrary projective scheme over a field $k$, and $\iota: X\hookrightarrow \mathbb{P}^n_k$ be an embedding. The sheaves
$$\bar{\iota}^*\mathcal{E}xt^i_{\mathcal{O}_{\mathbb{P}^n_k}}(\iota_*\mathcal{O}_X,\omega_{\mathbb{P}^n_k})$$ 
depend only on $X$ and $i$, but not on the embedding $\iota: X\hookrightarrow \mathbb{P}^n_k$.
\end{thm}

\begin{proof}[Proof]
Since both $X$ and $\mathbb{P}^n_k$ are schemes over $k$, we have the following commutative diagram
$$\xymatrix{
   X\ar[r]^{\iota} \ar[dr]^f & \mathbb{P}^n_k\ar[d]^g \\
    & \Spec(k)}$$
Since $\iota$ is finite and $g$ is smooth, all morphisms in the diagram are embeddable. Hence, 
\begin{align}
f^!(\mathcal{O}_{Spec(k)}) &\stackrel{{\rm by}\ (1)}{=}\iota^!g^!(\mathcal{O}_{\Spec(k)}) \notag\\
                           &\stackrel{{\rm by}\ (2)}{=}\iota^!(\omega_{\mathbb{P}^n_k})   \notag\\
                           &\stackrel{{\rm by}\ (3)}{=}\bar{\iota}^* \underline{\underline{R}}\ \underline{\Hom}(\iota_*\mathcal{O}_X, \omega_{\mathbb{P}^n_k})\notag
\end{align}
Since $\bar{\iota}^*$ is exact (page 165 in \cite{rd}), the sheaves 
$$\bar{\iota}^*\mathcal{E}xt^i_{\mathbb{P}^n_k}(\iota_*\mathcal{O}_X,\omega_{\mathbb{P}^n_k})$$
are the cohomology sheaves of $\bar{\iota}^* \underline{\underline{R}}\ \underline{\Hom}(\iota_*\mathcal{O}_X, \omega_{\mathbb{P}^n_k})$, 
i.e., the cohomology sheaves of $f^!(\mathcal{O}_{\Spec(k)})$, which do not depend on the embedding $\iota: X\hookrightarrow \mathbb{P}^n_k$.
\end{proof}

In the following remark $M_{(f)}$ denotes the homogeneous localization of a graded modules $M$ with respect to the multiplicative system $\{1,f,f^2,\dots\}$ (i.e., the degree-0 part of $M_f$), where $f$ is a homogeneous element; $\widetilde{M}$ denotes the sheaf on $X$ associated to $M$; and, 
$$\sideset{^*}{_S}\Hom(M,N):=\oplus_n\Hom_S(M,N)_n,$$
where $\Hom_S(M,N)_n$ is the set of homomorphsims of degree $n$ (see \S2 in \cite{ega2} for details; in \cite{ega2} $\sideset{^*}{}\Hom _S(M,N)$ is denoted simply $\Hom_S(M,N)$). If $M$ is finitely generated, $\sideset{^*}{}\Hom _S(M,N)$ coincides with $\Hom_S(M,N)$ in the usual sense. 

\begin{rem}{(2.5.12 in \cite{ega2})}
\label{ega2}
 Let $S$ be a graded noetherian ring, $M$ and $N$ two graded $S$-modules, and $f\in S_d$ ($d>0$).\par
One can define a canonical functorial $S_{(f)}$-modules homomorphism
$$\mu_f: (\sideset{^*}{_S}\Hom(M,N))_{(f)}\rightarrow \Hom_{S_{(f)}}(M_{(f)},N_{(f)}) $$
by sending $u/f^l$, where $u$ is a homomorphism of degree $ld$ to the homomorphism $M_{(f)}\rightarrow N_{(f)}$ which maps 
$x/f^m$ ($x\in M_{md}$) to $u(x)/f^{l+m}$.\par
For $g\in S_e$ ($e>0$), moreover, one has a commutative diagram
$$\begin{CD}
(\sideset{^*}{}\Hom_S(M,N))_{(f)} @>\mu_f>>        \Hom_{S_{(f)}}(M_{(f)},N_{(f)})  \\
      @VVV                                     @VVV                  \\
(\sideset{^*}{}\Hom_S(M,N))_{(fg)} @>\mu_{fg}>>    \Hom_{S_{(fg)}}(M_{(fg)},N_{(fg)})
\end{CD}$$
The vertical arrows are the canonical homomorphisms.\par
Furthermore, these $\mu_f$ define a canonical functorial 
homomorphism of $\mathcal{O}_X$-modules
$$\mu: \widetilde{\sideset{^*}{_S}\Hom(M,N)}\rightarrow \mathcal{H}om_{\mathcal{O}_X}(\tilde{M},\tilde{N}),$$
where $X=\Proj(S)$.
\end{rem}

We denote by $S_+$ the ideal of $S$ generated by the elements of positive degrees.
\begin{pro}{(Proposition 2.5.13 in \cite{ega2})}
 Suppose, for a graded noetherian ring $S$, the ideal $S_+$ is generated by $S_1$. Then 
$$\mu: \widetilde{\sideset{^*}{_S}\Hom(M,N)}\rightarrow \mathcal{H}om_{\mathcal{O}_X}(\tilde{M},\tilde{N})$$
is an isomorphism when $M$ is finitely generated and $X=\Proj(S)$. 
\end{pro}

The following proposition should be well-known to experts. Since we could not find a proof in the literature, we include a proof here.

\begin{pro}
\label{sheaf-ext}
Let $S$ be a graded Noetherian ring. Suppose that $S_+$ is generated by $S_1$. Let $M$ and $N$ be two graded $S$-modules. Then the homomorphism $\mu$ in Remark \ref{ega2} induces an isomorphism
$$\mu: \widetilde{\sideset{^*}{^i_S}\Ext(M,N)}\xrightarrow{\cong}\mathcal{E}xt^i_{\mathcal{O}_X}(\tilde{M},\tilde{N}),$$
when $M$ is finitely generated and $X=\Proj(S)$.
\end{pro}

\begin{proof}[Proof]
$M$ has a resolution with degree-preserving differentials 
$$\cdots\rightarrow F_r\xrightarrow{d_r}F_{r-1}\xrightarrow{d_{r-1}}\cdots\xrightarrow{d_2}F_1\xrightarrow{d_1}F_0\xrightarrow{d_0}M\rightarrow 0,$$
in which all $F_i$ are finitely generated graded free $S$-modules. Then we have an induced locally-free resolution for $\tilde{M}$
$$\cdots\rightarrow \tilde{F_r}\xrightarrow{\tilde{d}_r}\tilde{F}_{r-1}\xrightarrow{\tilde{d}_{r-1}}\cdots\xrightarrow{\tilde{d_2}}\tilde{F_1}\xrightarrow{\tilde{d_1}}\tilde{F_0}\xrightarrow{\tilde{d_0}}\tilde{M}\rightarrow 0,$$
where $\tilde{d}_i: \tilde{F}_i\rightarrow \tilde{F}_{i-1}$ is defined by
$$\tilde{d}_i|_{D_+(f)}:= d_{i(f)}: F_{i(f)}\rightarrow F_{i-1(f)}$$
since $$\tilde{F}_i(D_+(f))=F_{i(f)}\ {\rm and}\  \tilde{F}_{i-1}(D_+(f))=F_{i-1(f)}.$$
According to Proposition 6.5 in Chapter 3 in \cite{ag}, $\mathcal{E}xt^i_{\mathcal{O}_X}(\tilde{M},\tilde{N})$ is the $i$-th cohomology sheaf of the complex
\begin{equation}
\label{dag}
0\rightarrow \mathcal{H}om_{\mathcal{O}_X}(\tilde{M},\tilde{N})\xrightarrow{\partial_0}\mathcal{H}om_{\mathcal{O}_X}(\tilde{F_0},\tilde{N})\xrightarrow{\partial_1}\cdots\xrightarrow{\partial_r}
\mathcal{H}om_{\mathcal{O}_X}(\tilde{F_r},\tilde{N})\rightarrow \cdots,
\end{equation}
where $\partial_i$ is defined by 
$$\partial_i(U)(\phi)=\phi\circ\tilde{d}_i(U),\ \phi\in \mathcal{H}om_{\mathcal{O}_X}(\tilde{F}_{i-1},\tilde{N})(U)\ {\rm for\ an\ open\ subset}\ U.$$
Applying $\sideset{^*}{_S}\Hom(\cdot, N)$ to the resolution of $M$, we have
$$\cdots\rightarrow \sideset{^*}{_S}\Hom(M,N)\xrightarrow{\delta_0}\sideset{^*}{_S}\Hom(F_0,N)\xrightarrow{\delta_1}\cdots\xrightarrow{\delta_r}\sideset{^*}{_S}\Hom(F_r,N)\rightarrow\cdots,$$
where $\delta_i$ is defined by $$\delta_i(\varphi)=\varphi\circ d_i,\ \varphi\in \sideset{^*}{_S}\Hom(F_i,N).$$ 
This induces
\begin{equation}
\label{ddag}
0\rightarrow \widetilde{\sideset{^*}{_S}\Hom(M,N)}\xrightarrow{\widetilde{\delta_0}}\widetilde{\sideset{^*}{_S}\Hom(F_0,N)}\cdots\xrightarrow{\widetilde{\delta_r}}\widetilde{\sideset{^*}{_S}\Hom(F_r,N)}\rightarrow\cdots.
\end{equation}
Since $\widetilde{\cdot}$ is exact, 
$$\widetilde{\sideset{^*}{^i_S}\Ext(M,N)}=H^i(0\rightarrow \widetilde{\sideset{^*}{_S}\Hom(M,N)}\xrightarrow{\widetilde{\delta_0}}\widetilde{\sideset{^*}{_S}\Hom(F_0,N)}\cdots\xrightarrow{\widetilde{\delta_r}}\widetilde{\sideset{^*}{_S}\Hom(F_r,N)}\rightarrow\cdots)$$
Since the homomorphism $\mu:\widetilde{\sideset{^*}{_S}\Hom(\cdot,N)}\to \mathcal{H}om_{\mathcal{O}_X}(\tilde{\cdot},\tilde{N})$ is functorial, we have a commutative diagram
$$\begin{CD}
0 @>>>  \widetilde{\sideset{^*}{_S}\Hom(F_0,N)}  @>>> \widetilde{\sideset{^*}{_S}\Hom(F_1,N)} @>>> \cdots \\
@.       @V{\mu}VV                                          @V{\mu}VV \\
0 @>>>  \mathcal{H}om_{\mathcal{O}_X}(\tilde{F}_0,\tilde{N})  @>>> \mathcal{H}om_{\mathcal{O}_X}(\tilde{F}_1,\tilde{N}) @>>> \cdots
\end{CD}$$
in which the vertical arrows are isomorphisms. Therefore, the induced maps on homology are isomorphisms.
\end{proof}

Consider $R=k[x_0,\dots,x_n]$ as a graded ring with the standard grading. Let $N_1$ and $N_2$ be graded $R$-modules. The grading on $\sideset{^*}{}\Hom_R(N_1,N_2)$ is given by $\deg(\phi)=l$ for $\phi\in \Hom_R(N_1,N_2)_l$. This induces a grading on $\sideset{^*}{^t_R}\Ext(N_1,N_2)$ for all integers $t$. When $N_1$ is a finitely generated graded $R$-module, one has $\sideset{^*}{}\Hom_R(N_1,N_2)=\Hom_R(N_1,N_2),$
and hence 
$\sideset{^*}{^t_R}\Ext(N_1,N_2)=\Ext^t_R(N_1,N_2)$. 
Therefore, when $N_1$ is finitely generated, we will not distinguish $\sideset{^*}{}\Ext^t_R(N_1,N_2)$ and $\Ext^t_R(N_1,N_2)$, and just write $\Ext^t_R(N_1,N_2)$ (with the same grading on $\sideset{^*}{^t_R}\Ext(N_1,N_2)$ kept in mind). In particular, in what follows, we will write $$\Ext^{n+1-j}_R(R/I,R)$$ and $$\Ext^{n+1-i}_R(\Ext^{n+1-j}_R(R/I,R),R)$$ with $^*$ dropped.

\begin{thm} 
\label{inv}
Let $R=k[x_0,\dots,x_n]$ be the polynomial ring with the standard grading and $I$ be the defining ideal of a projective scheme $X$ under the embedding $\iota: X\hookrightarrow \mathbb{P}^n_k$. Set $\mathfrak{m}=(x_0,\dots,x_n)$ and
$$\cM=\Ext^{n+1-i}_R(\Ext^{n+1-j}_R(R/I,R),R).$$
Then $\cM_0$, the degree-0 piece of $\cM$, depends only on $X$, $i$, and $j$, but not on the embedding $\iota$. 
\end{thm}

\begin{proof}[Proof]
First we treat the case when $i\geq 2$.\\
Since $\omega_{\mathbb{P}^n_k}=\widetilde{R(-n-1)}$ and
$\iota_*\mathcal{O}_X=\widetilde{R/I}$, we have 
\begin{align}
\Ext^{n+1-i}_R(\Ext^{n+1-j}_R(R/I,R),R)_0
&= (\Ext^{n+1-i}_R(\Ext^{n+1-j}_R(R/I,R),R)(n+1-(n+1)))_0\notag\\
&= (\Ext^{n+1-i}_R(\Ext^{n+1-j}_R(R/I,R(-n-1)),R(-n-1)))_0\tag{i}\\
&\cong (\Hom_k(H^i_{\mathfrak{m}}(\Ext^{n+1-j}_R(R/I,R(-n-1)))_0,k))\tag{ii}\\
&\cong
\Hom_k(H^{i-1}(\mathbb{P}^n_k,\mathcal{E}xt^{n+1-j}_{\mathcal{O}_X}(\iota_*\mathcal{O}_X,\omega_{\mathbb{P}^n_k})),k)\tag{iii}\\
&\cong
\Hom_k(H^{i-1}(\mathbb{P}^n_k,\iota_*\bar{\iota}^*\mathcal{E}xt^{n+1-j}_{\mathcal{O}_X}(\iota_*\mathcal{O}_X,\omega_{\mathbb{P}^n_k})),k)\tag{iv}\\
&\cong
\Hom_k(H^{i-1}(X,\bar{\iota}^*\mathcal{E}xt^{n+1-j}_{\mathbb{P}^n_k}(\iota_*\mathcal{O}_X,\omega_{\mathbb{P}^n_k})),k)\tag{v} 
\end{align}
(i) holds because the graded modules $\Hom_R(N_1(-l),N_2)$, $\Hom_R(N_1,N_2)(l)$ and $\Hom_R(N_1,N_2(l))$ are all equal
for graded $R$-modules $N_1,N_2$ and integers $l$, hence
$\Ext^t_R(N_1(-l),N_2)=\Ext^t_R(N_1,N_2)(l)=\Ext^t_R(N_1,N_2(l))$ (13.1.9 in \cite{bs}).\\
(ii) follows from the Graded Local Duality for polynomial rings over a field (13.4.6 in \cite{bs}).\\
(iii) is a consequence of the fact that, for any graded $R$-module $M$,
$$(H^i_{\mathfrak{m}}(M))_0\cong H^{i-1}(\mathbb{P}^n_k,\tilde{M}),\ {\rm for}\ i\geq 2$$ (cf. Theorem A4.1
 in \cite{e}). It is here that we use $i\geq 2$.\\ 
(iv) holds because $\bar{\iota}^*$ is the restriction of $\iota^*$ to the subcategory of sheaves of $\iota_*\mathcal{O}_X$-modules, $\mathcal{E}xt^{n+1-j}_{\mathbb{P}^n_k}(\iota_*\mathcal{O}_X,\omega_{\mathbb{P}^n_k})$ is a sheaf of $\iota_*\mathcal{O}_X$-modules and $\iota_*\iota^*(\mathcal{F})=\mathcal{F}$ for any sheaf of $\iota_*\mathcal{O}_X$-module $\mathcal{F}$.\\
(v) follows from the fact that $\iota$ is a finite morphism.\\
This completes the proof in the case that $i\geq 2$. Next we treat the case when $i=0,1$.\par
First, we claim that $\Ext^{n+1-j}_R(R/I,R(-n-1))_0$ does not depend on the embedding. When $j\geq 2$, we reason as follows.\par
$\Ext^{n+1-j}_R(R/I,R(-n-1))_0$\\
$\cong \Hom_k(H^j_{\mathfrak{m}}(R/I),k)_0=\Hom_k(H^j_{\mathfrak{m}}(R/I)_0,k)$, by Graded Local Duality  (13.4.6 in \cite{bs})\\
$\cong \Hom_k(H^{j-1}(\mathbb{P}^n_k,\iota_*\mathcal{O}_X),k)$, by  Theorem A4.1 in \cite{e}\\
$\cong \Hom_k(H^{j-1}(X,\mathcal{O}_X),k)$, because $\iota$ is a finite morphism\\
This proves the claim in the case that $j\geq 2$.\\
When $j=0,1$, consider the exact sequence (cf. Theorem A4.1 in \cite{e})
$$0\to H^0_{\mathfrak{m}}(R/I)\to R/I\xrightarrow{\varphi}\bigoplus_{m\in\mathbb{Z}}H^0(\mathbb{P}^n_k,\widetilde{R/I}(m))\to H^1_{\mathfrak{m}}(R/I)\to 0.$$
The degree-0 piece of $\varphi$ is $\varphi_0:(R/I)_0=k\to H^0(\mathbb{P}^n_k,\widetilde{R/I})=H^0(X,\mathcal{O}_X)$. It is given by sending each element $c\in k=(R/I)_0$ to its associated global section on $X$, which is a string of the same constants $c\in k$, one for each connected component of $X$. Hence $\varphi_0$ does not depend on the embedding. Thus, the kernel and cokernel of $\varphi_0$, i.e., $H^0_{\mathfrak{m}}(R/I)_0$ and $H^1_{\mathfrak{m}}(R/I)_0$, do not depend on the embedding. Therefore, by Graded Local Duality (13.4.6 in \cite{bs})
$$\Ext^{n+1}_R(R/I,R(-n-1))_0\cong \Hom_k(H^0_{\mathfrak{m}}(R/I),k)_0=\Hom_k(H^0_{\mathfrak{m}}(R/I)_0,k)$$
and $$\Ext^n_R(R/I,R(-n-1))_0\cong \Hom_k(H^1_{\mathfrak{m}}(R/I),k)_0=\Hom_k(H^1_{\mathfrak{m}}(R/I)_0,k)$$
do not depend on the embedding. This completes the proof of the claim.\par
Now, let $T$ denote the module $\Ext^{n+1-j}_R(R/I,R(-n-1))$, for $0\leq j\leq d+1$. Again, by Theorem A4.1 in \cite{e}, we have an exact sequence
$$0\to H^0_{\mathfrak{m}}(T)\to T\xrightarrow{\varphi} \bigoplus_{m\in\mathbb{Z}}H^0(\mathbb{P}^n_k,\widetilde{T}(m))\to H^1_{\mathfrak{m}}(T)\to 0.$$
Notice that $\widetilde{T}=\iota_*\bar{\iota}^*\mathcal{E}xt^{n+1-j}_{\mathcal{O}_X}(\iota_*\mathcal{O}_X,\omega_{\mathbb{P}^n_k})$ by Proposition \ref{sheaf-ext}. Since $\iota$ is a finite morphism (a closed embedding), $H^0(\mathbb{P}^n_k,\widetilde{T})= H^0(X,\bar{\iota}^*\mathcal{E}xt^{n+1-j}_{\mathbb{P}^n_k}(\iota_*\mathcal{O}_X,\omega_{\mathbb{P}^n_k}))$, which is independent of the embedding by Theorem 2.1. By the claim, $T_0$ does not depend on the embedding. The degree-0 piece of the map $\varphi$ is  
$$\varphi_0:T_0\to H^0(\mathbb{P}^n_k,\widetilde{T})= H^0(X,\bar{\iota}^*\mathcal{E}xt^{n+1-j}_{\mathbb{P}^n_k}(\iota_*\mathcal{O}_X,\omega_{\mathbb{P}^n_k})),$$ 
which is given by sending each element of $T_0$ to its associated global section on $X$, hence $\varphi_0$ does not depend on the embedding. Thus, the kernel and cokernel of $\varphi_0$, i.e., $H^0_{\mathfrak{m}}(T)_0$ and $H^1_{\mathfrak{m}}(T)_0$, do not depend on the embedding. Therefore, 
$$\Ext^{n+1}_R(\Ext^{n+1-j}_R(R/I,R),R)_0\cong \Hom_k(H^0_{\mathfrak{m}}(T),k)_0$$
and $$\Ext^n_R(\Ext^{n+1-j}_R(R/I,R),R)_0\cong \Hom_k(H^1_{\mathfrak{m}}(T),k)_0$$
do not depend on the embedding. 
\end{proof}

\section{ The proof of the main theorem}
Throughout this section, $R$ will denote the polynomial ring $k[x_0,\dots,x_n]$ where $k$ is of characteristic $p>0$. When $R$ is considered a graded ring, the grading is always the standard one, i.e., $\deg(x_i)=1$ and $\deg(a)=0$ for all $a\in k$. By an $R\{f\}$-module $M$, we mean an $R$-module $M$ equipped with an action of Frobenius, i.e. a map of abelian groups $f: M\to M$, such that $f(rm)=r^pf(m)$, for all $r\in R,m\in M$ ($f$ is also called a $p$-linear endomorphism on $M$). \par

Since any field extension will not affect Lyubeznik numbers (i.e., if $I\subset k[x_0,\dots,x_n]$ is an ideal, $K$ is a field extension of $k$ and $A_K=(K[x_0,\cdots,x_n]/IK[x_0,\cdots,x_n])_{(x_0,\dots,x_n)}$, then $\lambda_{ij}(A)= \lambda_{ij}(A_K)$), we may and we do assume that the underlying field $k$ is algebraically closed.\par
First, we collect some material (from Section 2 in \cite{l5}) about the Frobenius functor
$$F:R\text{-mod}\to R\text{-mod}.$$
Let $R^{(1)}$ be the additive group of $R$ regarded as an $R$-bimodule with usual left $R$-action and with right $R$-action given by $r'\cdot r=r^pr'$ for all $r\in R,r'\in R^{(1)}$. The Frobenius functor $F$ is defined by 
$$F(M)=R^{(1)}\otimes_RM$$
$$F(M\xrightarrow{\phi}N)=(R^{(1)}\otimes_RM\xrightarrow{id\otimes\phi}R^{(1)}\otimes_RN)$$
for all $R$-modules $M,N$ and all $R$-module homomorphisms $\phi$, where $F(M)$ acquires its $R$-module structure via the left $R$-module structure on $R^{(1)}$. Notice that $F$ is an exact functor by Kunz's theorem (Theorem 2.1 in \cite{kunz}).\par
If $M$ is a free $R$-module with an $R$-basis $\{b_1,b_2,\dots\}$, then the $R$-homomorphism $F(M)\to M$ given by 
$$\sum_i r_i\otimes s_ib_i\mapsto \sum_i r_is^p_ib_i$$
is an isomorphism.\par
For any $R$-module $M$, the isomorphism
$$F(\Hom_R(M,R))\to \Hom_R(F(M),R)$$
is given by $r\otimes \phi\mapsto \psi$ where $\psi$ is defined by $\psi(s\otimes m)=rs\phi(m)^p$ for all $r,s\in R^{(1)},m\in M, \phi\in \Hom_R(M,R)$.\par

\begin{rem}{\bf (grading on $F(M)$)}
\label{grading}
In order to make the isomorphism $F(R)\to R$ (defined above) a degree-preserving $R$-homomorphism, the only grading one can put on $F(R)$ is given by 
$$\deg(r'\otimes r)=\deg(r')+p\deg(r)$$
where $r'\in R^{(1)}$ and $r\in R$ are homogeneous elements. In general, for any graded $R$-module $M$, the grading on $F(M)$ is given by 
$$\deg(r'\otimes m)=\deg(r')+p\deg(m)$$
where $r'\in R^{(1)}$ and $m\in M$ are homogeneous elements. It is easy to check that, under this grading, the isomorphism $F(\Hom_R(M,R))\to \Hom_R(F(M),R)$ (defined above) is degree-preserving, hence, so are induced isomorphisms
$$F(\Ext^i_R(M,R))\to \Ext^i_R(F(M),R).$$ 
In particular, for any homogeneous ideal $I$ of $R$, the isomorphisms $$F(\Ext^i_R(R/I,R))\to \Ext^i_R(F(R/I),R)$$ are degree-preserving. Similarly, the induced isomorphisms
\begin{align}
F(Ext^i_R(\Ext^j(R/I,R),R)) &\to \Ext^i_R(F(\Ext^j_R(R/I,R)),R)\notag \\
                            &\to \Ext^i_R(\Ext^j_R(F(R/I),R),R)\notag
\end{align}
are degree-preserving for all $i,j$.
\end{rem}

Let $S$ be a graded commutative ring, then one can define $\sideset{^*}{}\varprojlim$ as a graded version of $\varprojlim$ of an inverse system of graded $S$-modules as follows. Let $\{M_i,\theta_{ji}:M_j\to M_i\}$ be an iverse system of graded $S$-modules where $\theta_{ji}$ are degree-preserving $S$-module homomorphisms. Define $$(\sideset{^*}{}\varprojlim_iM_i)_n=\varprojlim_i(M_i)_n,$$
where $\varprojlim_i(M_i)_n$ is the inverse limit of the inverse system $\{(M_i)_n,(\theta_{ji})_n:(M_j)_n\to (M_i)_n\}$. Then define
$$\sideset{^*}{}\varprojlim_iM_i=\oplus_{n\in \mathbb{Z}}(\sideset{^*}{}\varprojlim_iM_i)_n.$$
Then $\sideset{^*}{}\varprojlim_iM_i$ is naturally a graded $S$-module.\par
In general, $\sideset{^*}{}\varprojlim$ and $\varprojlim$ are different. 

\begin{ex}
\label{example}
Let $S=k[x_0,\dots,x_n]$ be a polynomial ring in $n+1$ variables over a field $k$ with $\Char(k)=p>0$ and $I$ be a proper homogeneous ideal of $S$. If $M_i=S/I^{[p^i]}$, then 
$$\sideset{^*}{}\varprojlim_iM_i=S,\ {\rm but}\ \varprojlim_iM_i=\hat{S}^I,$$
where $\hat{S}^I$ is the $I$-adic completion of $S$.
\end{ex}
\begin{proof}[Proof]
The second assertion, $\varprojlim_iM_i=\hat{S}^I$, is clear.\par
To prove that $\sideset{^*}{}\varprojlim_iM_i=S$, it suffices to show
$$\varprojlim_i(M_i)_n=S_n$$
for every integer $n\geq 0$.\par
For each $n\geq 0$, there exists an integer $i(n)$ such that $p^j>n$ when $j\geq i(n)$. Then $(M_j)_n=S_n$ for all $j\geq i(n)$, and hence the inverse system $\{(M_i)_n\}$ is indeed the following
\begin{equation}
\label{sn}
(S/I^{[p]})_n\leftarrow (S/I^{[p^2]})_n\leftarrow\cdots\leftarrow S_n\xleftarrow{=}S_n\xleftarrow{=}\cdots
\end{equation}
It is clear that the inverse limit of (\ref{sn}) is $S_n$.
\end{proof}
One may also define $\sideset{^*}{}\varinjlim$ as a graded version of $\varinjlim$ similarly. However, by doing so, one does not get anything new, since (it is eay to check that) $\varinjlim_iM_i$ is naturally graded with respect to the grading on $M_i$'s whenever all $M_i$ are graded and $\theta_{ij}$ are degree-preserving.\par

Let $T$ be a subset of a $R$-module $M$, throughout this section $\langle
T\rangle_R$ will denote the $R$-submodule of $M$ generated by all the elements of $T$.\par
Let $M$ be a graded $R$-module. If the action of Frobenius $f$ on $M$ satisfies $\deg(f(m))=p\deg(m)$, then $f$ sends $M_0$ to itself and 
$$M_s:= \bigcap^{\infty}_{i=1}f^i(M)$$ is contained in $M_0$. If the degree-0 piece $M_0$ of $M$ is a finite $k$-space, so is $M_s$, and hence $N=\langle M_s\rangle_R$ is finitely generated. \par
Let $F$ denote the natural Frobenius functor on the category of $R$-modules. The action of Frobenius $f:M\to M$ on $M$ induces an $R$-module homomorphism $$\beta:F(M)\to M$$ given by $\beta(r\otimes m)=rf(m)$. There results an inverse system
$$M\xleftarrow{\beta} F(M)\xleftarrow{F(\beta)}F^2(M)\leftarrow\cdots\leftarrow F^i(M)\xleftarrow{F^i(\beta)}F^{i+1}(M)\leftarrow\cdots.$$

\begin{lem}
\label{limit}
Let $R=k[x_0,\dots,x_n]$ be the polynomial ring over $k$. Let $M$ be an $R\{f\}$-module and a finitely generated graded $R$-module. Assume that the action of Frobenius $f$ on $M$ satisfies $$\deg(f(m))=p\deg(m),$$
for all homogeneous $m\in M$. Let $N=\langle M_s\rangle_R$, where $M_s$ is the stable part of $M$. Then $N$ is an $R\{f\}$-submodule of $M$. Let $\beta:F(M)\to M$ be the $R$-module homomorphism defined above.  Then $\beta$ is degree-preserving with respect to the grading in Remark \ref{grading} and 
$$\sideset{^*}{}\varprojlim_{i\to\infty} F^i(N)=\sideset{^*}{}\varprojlim_{i\to\infty}F^i(M).$$
\end{lem}
\begin{proof}[Proof]
If $\sum_ja_jn_j\in N$ where $a_j\in R$ and $n_j\in M_s$, then $f(\sum_ja_jn_j)=\sum_ja^p_jf(n_j)\in N$ since $n_j\in M_s$, hence $N$ is indeed an $R\{f\}$-submodule of $M$. Since $\deg(f(m))=p\deg(m)$, one has 
$$\deg(\beta(r\otimes m))=\deg(rf(m))=\deg(r)+\deg(f(m))=\deg(r)+p\deg(m)=\deg(r\otimes m),$$
i.e. $\beta$ is degree-preserving with respect to the grading in Remark \ref{grading}. And so are the induced homomorphisms $F^i(\beta)$.\par
Since $F$ is exact, the exact sequence of $R\{f\}$-modules
$$0\to N\to M\to M/N\to 0$$
induces an exact sequence of inverse systems
$$\begin{CD}
@.   @AAA     @AAA    @AAA\\
 0   @>>>   F^i(N)_n @>>>  F^i(M)_n @>F^i(\pi)_n>>  F^i(M/N)_n @>>> 0 \\
@.  @AAF^i(\beta|_{F(N)})_nA   @AAF^i(\beta)_nA   @AAF^i(\bar{\beta})_nA\\
 0  @>>>   F^{i+1}(N)_n  @>>>   F^{i+1}(M)_n @>>F^{i+1}(\pi)_n>  F^{i+1}(M/N)_n @>>> 0\\
 @.  @AAA   @AAA   @AAA\\
\end{CD} $$
Since $M$ is a finitely generated graded $R$-module, $M_0$ is a finite $k$-space, hence $M_s$ is also a finite $k$-space. Since $N=\langle M_s\rangle_R$ is finitely generated, so are $F^i(N)$, and hence each $F^i(N)_n$ is a finite $k$-space, so the inverse system 
$\{F^i(N)_n\}$ satisfies the Mittag-Leffler condition (page 191 in \cite{ag}), we have (by Proposition 9.1 on page 192 in \cite{ag}) an exact sequence
$$0\to \varprojlim F^i(N)_n\to \varprojlim F^i(M)_n\to \varprojlim F^i(M/N)_n\to 0$$
for each $n$, which induces another exact sequence
\begin{equation}
\label{ex-seq}
0\to \sideset{^*}{}\varprojlim F^i(N)\to \sideset{^*}{}\varprojlim F^i(M)\to \sideset{^*}{}\varprojlim F^i(M/N)\to 0.
\end{equation}

We claim that there exists an integer $l$ such that $F^i(M/N)\to\cdots\to M/N$ (or briefly, $F^i(M/N)\to M/N$) has 0 image in degrees $\leq 1$ for all $i\geq l$ and we reason as follows. The homomorphism $F^i(M/N)\to M/N$ is given by $r\otimes \bar{m}\mapsto \overline{rf^i(m)}$ for $r\in R^{(i)}\cong R$ and $m\in M$, thus, 
$$\im(F^i(M/N)\to M/N)=\langle f^i(M/N)\rangle_R.$$
Since $M/N$ is finitely generated, there exists an integer $l_0$ such that $f^i(\bar{m})=0$ if $\deg(\bar{m})<0$ and $i\geq l_0$. This tells us that $$\langle f^i(M/N)\rangle_R=\langle f^i((M/N)_{\geq 0})\rangle_R=\langle f^i((M/N)_0)\rangle_R+\langle f^i((M/N)_{>0})\rangle_R,\ {\rm for}\ i\geq l_0.$$
Since $\deg(f(m))=p\deg(m)$, we get that $\langle f^i((M/N)_{>0})\rangle_R\subset (M/N)_{\geq p^i}$ which does not have any element sitting in $(M/N)_{\leq 1}$. Since $k$ is algebraically closed, we have a descending chain of $k$-subspaces
$$f(M_0)\supset f^2(M_0)\supset\cdots\supset f^i(M_0)\supset f^{i+1}(M_0)\supset \cdots$$
Since $M_0$ is a finite $k$-space, this descending chain stabilizes, i.e. there exists an integer $t$ such that 
$$f^t(M_0)=f^{t+1}(M_0)=\cdots.$$
Hence $N=\langle f^t(M_0)\rangle_R$. Thus, if $i\geq t$, then $f^i((M/N)_0)=0$. Therefore, we can pick $l=\max\{l_o,t\}$. This finishes the proof of our claim.\par
Followed directly from our claim, the map $F^{i+j}(M/N)\to\cdots\to F^j(M/N)$ has 0 image in degrees $\leq p^j$ for all $j$ and $i\geq l$. Since the maps in the inverse system are degree-preserving, $\varprojlim_i(F^i(M/N))_n=0$ for all $n$, hence $\sideset{^*}{_i}\varprojlim(F^i(M/N))=0$. By the short exact sequence (\ref{ex-seq}), one has 
$$\sideset{^*}{_i}\varprojlim F^i(M)=\sideset{^*}{_i}\varprojlim F^i(N).$$
 \end{proof}

We need the following result from \cite{lcd}.
\begin{lem}{(Lemma 1.14 in \cite{lcd})}
\label{basis}
Let $k$ be an algebraically closed field of characteristic $p>0$, and let $V$ be
a finite dimensional $k$-vector space with a bijective $p$-linear endomorphism
$f$. Then there is a basis $e_1,\cdots,e_r$ of $V$ with $f(e_i)=e_i$ for each $i$.
\end{lem}
\begin{pro}
\label{limit-stable}
Let $M,N$ be as in Lemma \ref{limit}. Let $M_s$ denote the stable part of $M$. Then 
$$\sideset{^*}{}\varprojlim_iF^i(N)=R^s,$$ where
$s=\dim_k(M_s)$.
\end{pro}
\begin{proof}[Proof]
Use induction on $s$. Notice that $N$ is generated by $s$ element and clearly $f$ acts bijectively on $M_s$.\par
When $s=1$, $N$ is generated by a single element $a\in M_0$ with $f(a)=a$ by Lemma \ref{basis}. Then $N\cong R/I$,
where $I=\Ann_R(a)$. In this case, the map $\beta:F(N)\to N$ is given by
$1\otimes a\mapsto 1\cdot f(a)=a$. Let $\varphi$ denote the isomorphism
$N\xrightarrow{\sim}R/I$ and $\phi_j$ denote the isomorphism
$F^j(R/I)\xrightarrow{\sim}R/I^{[p^j]}$. It is not hard to see that we have the
following commutative diagram
$$\begin{CD}
  F^i(N)     @<<<     F^{i+1}(N)\\
 @VVF^i(\varphi)V      @VVF^{i+1}(\varphi)V\\
 F^i(R/I)    @<<<     F^{i+1}(R/I)\\
  @VV\phi_iV          @VV\phi_{i+1}V\\
 R/I^{[p^i]} @<<<     R/I{[p^{i+1}]}
\end{CD}$$

Hence, we have a map between two graded inverse systems
$$\begin{CD}
 N  @<<<      F(N)  @<<< F^2(N)   @<<<  ... \\
@VV\varphi V  @VV\varphi_1 V   @VV\varphi_2 V\\
R/I @<<<     R/I^{[p]}  @<<<   R/I^{[p^2]} @<<< ... 
\end{CD}$$
Since $F^j(\varphi)$ and $\phi_j$ are degree-preserving isomorphisms, we have
$$\sideset{^*}{}\varprojlim_{i\to\infty}F^i(N)\cong\sideset{^*}{}\varprojlim_{i\to\infty}R/I^{[p^i]}= R,$$
where the last equality follows from Example \ref{example}.\par
Assume that $s\geq 2$ and the proposition is true for $s-1$. Let
$a_1,\cdots,a_s$ denote the generators of $N$ with $f(a_i)=a_i$. They exist by Lemma \ref{basis}. Let $N'$ be
the submodule of $N$ generated by $a_s$. Then,
$$\sideset{^*}{}\varprojlim_{i\to\infty}(F^i(N')\cong R.$$
Since $\beta$ and $F^i(\beta)$ are degree preserving, we have the following exact sequence of inverse systems
$$\begin{CD}
@.   @AAA   @AAA    @AAA   \\
0   @>>>  F^i(N')_n  @>>>  F^i(N)_n  @>>> 
F^i(N/N')_n  @>>> 0\\
@.       @AAA         @AAA           @AAA\\
0   @>>>  F^{i+1}(N')_n @>>>  F^{i+1}(N)_n  @>>> 
F^{i+1}(N/N')_n @>>> 0\\
@.   @AAA   @AAA    @AAA    
\end{CD}$$ 
Since $N'$ is finitely generated, so are $F^i(N')$, and hence each $F^i(N')_n$ is a finite $k$-space. Thus, the inverse system $\{F^i(N')_n\}$ satisfies
the Mittag-Leffler condition (page 191 in \cite{ag}). So (by Proposition 9.1 on page 192 in \cite{ag})
$$0\to \varprojlim F^i(N')_n\to \varprojlim
F^i(N)_n\to \varprojlim F^i(N/N')_n\to 0$$ 
is exact for every $n$. And hence,
$$0\to \sideset{^*}{}\varprojlim F^i(N')\to \sideset{^*}{}\varprojlim
F^i(N)\to \sideset{^*}{}\varprojlim F^i(N/N')\to 0$$
is an exact sequence. By the inductive hypothesis, $\sideset{^*}{}\varprojlim
F^i(N/N')\cong R^{s-1}$ which is projective over
$R$, and hence the exact sequence splits,
$$\sideset{^*}{}\varprojlim F^i(N)\cong \sideset{^*}{}\varprojlim
F^i(N')\oplus R^{s-1}\cong R\oplus R^{s-1}\cong
R^s.$$ 
\end{proof}

\begin{pro}
\label{f-on-ext}
Let $\cM=\Ext^{n+1-i}_R(\Ext^{n+1-j}_R(R/I,R),R)$. There is a natural action of Frobenius $f$ on $\cM$ satisfying $$\deg(f(m))=p\deg(m),$$ for all homogeneous element $m\in \cM$.
\end{pro}
\begin{proof}
We construct an action of Frobenius on
$$\cM=\Ext^{n+1-i}_R(\Ext^{n+1-j}_R(R/I,R),R)$$ as follows. Let $\alpha:\cM\to
F(\cM)$ be defined by $m\mapsto 1\otimes m$ for all $m\in \cM$, then $\alpha$ is a $p$-linear map, because $$\alpha(rm)=1\otimes rm=1\cdot r\otimes m=r^p\cdot 1\otimes m=r^p\alpha_1(m).$$ Let $\alpha_1:F(\cM)\to
\Ext^{n+1-i}_R(\Ext^{n+1-j}_R(R/I^{[p]},R),R)$ be the isomorphism derived from
the flatness of $R^{(1)}$ and the isomorphism $F(R/I)\cong R/I^{[p]}$. Since the isomorphism
$$F(R/I)\xrightarrow{\sim}R/I^{[p]},\ r'\otimes \bar{r}\mapsto \overline{r'r^p}$$
is degree-preserving with respect to the grading introduced in Remark \ref{grading}, $\alpha_1$ is also degree-preserving. 
Let $$\alpha_2:\Ext^{n+1-i}_R(\Ext^{n+1-j}_R(R/I^{[p]},R),R)\to \cM$$ be the
homomorphism induced by the surjection $R/I^{[p]}\to R/I$. Notice that $\alpha_2$ is also degree-preserving when $\Ext^{n+1-i}_R(\Ext^{n+1-j}_R(R/I^{[p]},R),R)$ and $\cM$ are considered as graded $R$-modules with the usual grading, i.e., the grading induced from the standard grading on $R$, since the natural surjection $R/I^{[p]}\to R/I$ is degree-preserving. Then
$$f:=\alpha_2\circ\alpha_1\circ\alpha$$ is an action of Frobenius on $\cM$, since $\alpha$ is $p$-linear and $\alpha_1,\alpha_2$ are $R$-linear.\par
Since $\deg(r'\otimes m)=\deg(r')+p\deg(m)$ where $r'\in R^{(1)}$ and $m\in \cM$ are homogeneous elements, one has
$$\deg(\alpha(m))=\deg(1\otimes m)=p\deg(m).$$ But $\alpha_1$ and $\alpha_2$ are degree-preserving, thus 
$$\deg(f(m))=\deg(\alpha_2\circ\alpha_1\circ\alpha(m))=\deg(\alpha(m))=p\deg(m).$$
\end{proof}

It is also worth pointing out that the homomorphism $\beta: F(\cM)\to\cM$ induced by the action of Frobenius on $\cM$ is actually $\alpha_2\circ\alpha_1$ in the proof of Proposition \ref{f-on-ext}.\par

\begin{pro}
\label{graded-inverse}
Let $S$ be a graded commutative ring. Let $N$ be a graded $S$-module and $\{M_i,\theta_{ij}:M_i\to M_j\}$ be a direct system of graded $S$-modules where $\theta_{ij}$ are degree-preserving $S$-module homomorphisms. Then
$$\sideset{^*}{_S}\Hom(\varinjlim_iM_i,N)\cong \sideset{^*}{}\varprojlim_i\sideset{^*}{_S}\Hom(M_i,N).$$
\end{pro}
\begin{proof}[Proof]
It suffices to show that 
$$\Hom_S(\varinjlim_iM_i,N)_n\cong \varprojlim_i(\Hom_S(M_i,N)_n)$$
for every integer $n$.\par
Let $\theta_i$ denote the natural (degree-preserving) map $M_i\to \varinjlim_iM_i$. Define $$g:\Hom_S(\varinjlim_iM_i,N)_n\to
\varprojlim_i(\Hom_S(M_i,N)_n)$$ as follows. For any
$\varphi\in\Hom_S(\varinjlim_iM_i,N)_n$,
$g(\varphi):=(\cdots,\varphi\circ\theta_i,\cdots)$. Since
$\varphi\circ\theta_j\circ\theta_{ij}=\varphi\circ\theta_i$ and the transition
homomorphisms in the inverse system $\{\Hom_S(M_i,N)_n\}$ are given by composing
with $\theta_{ij}$ which are degree-preserving, $g$ is well-defined.\par
First we prove that $g$ is an injection.  Assume that $\varphi\in \Hom_S(\varinjlim_iM_i,N)_n$ is not 0 and $g(\varphi)=0$. Then there exists a nonzero element $y\in \varinjlim_iM_i$ such that $\varphi(y)\neq 0$. Since $\{M_i\}$ is a direct system, there exist an index $i$ and an element $y_i\in M_i$ such that $\theta_i(y_i)=y$. Then
$\varphi\circ\theta_i(y_i)=\varphi(y)\neq 0$, hence $\varphi\circ\theta_i\neq 0$, so
$g(\varphi)\neq 0$, a contradiction.\par
Next we prove that $g$ is a surjection. Let $(\cdots,\varphi_i,\cdots)$ be an
arbitrary element in $\varprojlim_i(\Hom_S(M_i,N)_n)$, i.e., $\varphi_i$ satisfy
$\varphi_i=\varphi_j\circ\theta_{ij}$ for all $j\geq i$. From the universal property of $\varinjlim_iM_i$, there exists a (unique) homomorphism
$\varphi\in \Hom_S(\varinjlim_iM_i, N)_n$ such that $\varphi\circ\theta_i=\varphi_i$, hence
$g(\varphi)=(\cdots,\varphi_i,\cdots)$.
\end{proof}

Let $R=k[x_0\dots,x_n]$ be the polynomial ring over $k$ and $\mathfrak{m}=(x_0,\dots,x_n)$. Let $\E$ denote the injective hull of $R/\mathfrak{m}$ and $\sideset{^*}{}\E$ denote the $^*$injective hull of $R/\mathfrak{m}$ (cf. 13.2.1 in \cite{bs}).
\begin{pro}
\label{limit-lambda}
Let $\cM=\Ext^{n+1-i}_R(\Ext^{n+1-j}_R(R/I,R),R)$. Then 
$$\sideset{^*}{}\varprojlim_eF^e(\cM)\cong R^{\lambda_{i,j}(A)}.$$
\end{pro}
\begin{proof}[Proof]
\begin{align}
R^{\lambda_{i,j}(A)} &\cong \sideset{^*}{_R}\Hom(\sideset{^*}{^{\lambda_{i,j}(A)}}\E,\sideset{^*}{}\E)\tag{i}\\
                           &\cong \sideset{^*}{_R}\Hom(H^i_{\mathfrak{m}}(H^{n+1-j}_I(R))(-n-1),\sideset{^*}{}\E)\tag{ii}\\
                           &\cong \sideset{^*}{_R}\Hom(H^i_{\mathfrak{m}}(\varinjlim_e\Ext^{n+1-j}_R(R/I^{[p^e]},R))(-n-1),\sideset{^*}{}\E)\tag{iii}\\
                           &\cong \sideset{^*}{_R}\Hom(\varinjlim_eH^i_{\mathfrak{m}}(\Ext^{n+1-j}_R(R/I^{[p^e]},R(-n-1))),\sideset{^*}{}\E)\tag{iv}\\
                           &\cong \sideset{^*}{_e}\varprojlim\sideset{^*}{_R}\Hom(H^i_{\mathfrak{m}}(\Ext^{n+1-j}_R(R/I^{[p^e]},R(-n-1)))),\sideset{^*}{}\E)\tag{v}\\
                          &\cong \sideset{^*}{_e}\varprojlim\Ext^{n+1-i}_R(\Ext^{n+1-j}_R(R/I^{[p^e]},R(-n-1)),R(-n-1))\tag{vi}\\
                         &\cong \sideset{^*}{_e}\varprojlim(\Ext^{n+1-i}_R(\Ext^{n+1-j}_R(R/I^{[p^e]},R),R))\notag\\
                         &\cong \sideset{^*}{_e}\varprojlim(F^e(\cM))\notag
\end{align}
(i) follows from Graded Matlis Duality (13.4.5 in \cite{bs}).\\
(ii) holds because $H^i_{\mathfrak{m}}(H^{n+1-j}_I(R))\cong \E^{\lambda_{i,j}(A)}$ by Lemma 2.2 in \cite{l3} and $\sideset{^*}{}\E=\E(-n-1)$ by 13.3.9 in \cite{bs}.\\
(iii) follows from the definition of local cohomology.\\
(iv) holds since local cohomology commutes with direct limits and degree shifting.\\
(v) follows from Proposition \ref{graded-inverse}.\\
(vi) is a consequence of Graded Local Duality (13.4.2 in \cite{bs}).

\end{proof}

\begin{proof}[Proof of Main Theorem]
We still use $\cM$ to denote $\Ext^{n+1-i}_R(\Ext^{n+1-j}_R(R/I,R),R)$. From Proposition \ref{f-on-ext}, we know that $\cM$ satisfies the hypothses in Lemma \ref{limit}, hence, by Proposition \ref{limit-stable} and Proposition \ref{limit-lambda}, we have 
$$\lambda_{i,j}(A)=\dim_k(\cM_s).$$ 
\end{proof}

\section{Acknowledgment}
This paper is from my thesis, I would like to thank my advisor, Professor Gennady Lyubeznik, for his support and guidance. I also want to thank Professor Joseph Lipman for pointing out reference \cite{gr} to me.
  

\begin{thebibliography}{9}
 \bibitem{bb} M. Blickle and R. Bondu, \emph{local cohomology multiplicities in terms of \'etale cohomology}, Ann. Inst. Fourier (Grinoble) \textbf{55} (2005), no.7, 2239-2256. 
 \bibitem{bs}
   M. Brodman and R. Sharp, \emph{Local Cohomology}, Cambridge University Press, 1998.
 
\bibitem{e}
   D. Eisenbud, \emph{Commutative algebra}, Graduate Texts in Mathematics \textbf{150}, 1995. 
\bibitem{ls}
   R. Garcia Lopez and C. Sabbah, \emph{Topological computation of local cohomology multiplicities}, Collect. Math. \textbf{49} (1998), no.2-3, 317-324.  

\bibitem{ega1}
   A. Grothendieck, \emph{\'El\'ements de g\'eom\'etrie alg\'ebrique: I. Le langage des sch\'emas}, Publ. Math. I.H.\'E.S, No. 4 (1960), 5-228.
 \bibitem{ega2}
    A. Grothendieck, \emph{\'El\'ements de g\'eom\'etrie alg\'ebrique: II. \'Etude globale \'el\'ementaire de quelques classes de morphismes}, Publ. Math. I.H.\'E.S, No. 8 (1961), 5-222.
 \bibitem{gr} 
A. Grothendieck, \emph{Th\'eor\`emes de dualit\'e pour les faisceaux alg\'ebrques coh\'erents}, S\'eminaire Bourbaki (1957), no.149, 1-25. 

\bibitem{rd}
    R. Hartshorne, \emph{Residues and duality}, Lecture Notes in mathematics \textbf{20}, 1966. 
 \bibitem{ag}
    R.Hartshorne, \emph{Algebraic geometry}, Graduate Texts in Mathematics \textbf{52}, 1977. 

\bibitem{lcd}
    R.Hartshorne and R. Speiser, \emph{Local cohomological dimension in characteristic $p$}, Ann. Math. (2), \textbf{1} (1977), 45-79.

\bibitem{kunz}
   E. Kunz, \emph{Characterisation of regular local rings of characteristic p}, Amer. J. of Math., \textbf{91} (1969) 772-784.
\bibitem{k1}
   K. I. Kawasaki, \emph{On the Lyubeznik number of local cohomology modules}, Bull. Nara Univ. Ed. Natur. Sci. \text{49} (2000) no.2, 5-7.
\bibitem{k2}
   K. I. Kawasaki, \emph{On the highest Lyubeznik number}, Math. Proc. Cambr. Phil. Soc., \textbf{132} (2002), no.3, 409-417.

\bibitem{l1}
   G. Lyubeznik, \emph{Finiteness properties of local cohomology modules}, Invent. Math. \textbf{113} (1993), no.1, 41-55.
\bibitem{l2}
   G. Lyubeznik, \emph{A partial survey of local cohomology}, Local cohomology and its applications (Guanajuato, 1999), 121-154, Lecture Notes in Pure and Appl. Math., \textbf{226}, 2002.
\bibitem{l3}   G. Lyubeznik, \emph{On some local cohomology invariants of local rings}, Math. Z. \textbf{254} (2006), no.3, 627-640.
\bibitem{l5}
  G. Lyubeznik, \emph{On the vanishing of local cohomology in characteristic $p>0$}, Compos. Math. \textbf{142} (2006), no.1, 207-221.
 \bibitem{w}
   U. Walther, \emph{On the Lyubeznik numbers of a local ring}, Proceedings of the AMS, \textbf{129} (6) (2001), 1631-1634.
 \bibitem{zh}
   W. Zhang, \emph{On the highest Lyubeznik number of a local ring}, Compos. Math. \textbf{143} (2007), no.1, 82-88. 
 \end{thebibliography}
\end{document}